\documentclass[11pt]{amsart}
\usepackage{amsmath}
\usepackage{amsfonts}
\usepackage{amssymb}
\usepackage{amsthm}
\usepackage{url}
\usepackage[all]{xy}
\usepackage{dsfont}
\usepackage{graphicx}
\usepackage{caption}
\usepackage{subcaption}
\usepackage{comment} 
\usepackage{stmaryrd}
\usepackage{hyperref}
\usepackage{todonotes}
\usepackage{color}
\usepackage{enumerate,tikz-cd}
\usepackage[backend=biber,style=alphabetic,url=false,doi=false,isbn=false,sorting=nyt,maxbibnames=9,giveninits=true]{biblatex}
\usepackage{MnSymbol}
\usepackage{comment}
\usepackage{mathrsfs}
\allowdisplaybreaks
\usepackage{geometry}
\numberwithin{equation}{section}

\newtheorem{proposition}{Proposition}[section]

\newtheorem{lemma}[proposition]{Lemma}
\newtheorem{theorem}[proposition]{Theorem}
\newtheorem{corollary}[proposition]{Corollary}

\theoremstyle{definition}

\newtheorem{definition}[proposition]{Definition}

\DeclareMathOperator{\Supp}{Supp}

\DeclareMathOperator{\codim}{codim}

\DeclareMathOperator{\rk}{rk}

\DeclareMathOperator{\DF}{DF}

\DeclareMathOperator{\ord}{ord}

\newcommand{\R}{\mathbb{R}}

\newcommand{\A}{\mathbb{A}}

\newcommand{\Q}{\mathbb{Q}}
\newcommand{\G}{\mathbb{G}}

\newcommand{\PP}{\mathbb{P}}
\newcommand{\QQ}{\mathbb{Q}}
\newcommand{\OO}{\mathcal{O}}

\renewcommand{\epsilon}{\varepsilon}

\newcommand{\M}{\mathcal{M}}

\newcommand{\J}{\mathcal{J}}

\newcommand{\F}{\mathcal{F}}
\newcommand{\B}{\mathcal{B}}

\renewcommand{\L}{\mathcal{L}}
\newcommand{\X}{\mathcal{X}}
\newcommand{\Y}{\mathcal{Y}}

\newcommand{\W}{\mathcal{W}}

\renewcommand{\phi}{\varphi}

\newcommand\vol{\mathrm{vol}}

\newcommand\cA{\mathcal{A}}

\newcommand\cF{\mathcal{F}}

\newcommand\cN{\mathcal{N}}
\newcommand\cO{\mathcal{O}}

\title{Singularity criteria for K-stability of adjoint foliated structures}

\author[T. S. Papazachariou]{Theodoros Stylianos Papazachariou}
\address{Yau Mathematical Sciences Center, Jingzhai, Tsinghua University, Haidian District, Beijing, China.}
\email{tpapazachariou@mail.tsinghua.edu.cn}

\begin{document}

\begin{abstract}
    We prove singularity criteria for the $t$-K-stability of adjoint foliated structures. We first show that K-semistability of adjoint foliated structures implies log canonicity by extending Odaka's flag ideal characterisation of the mixed Donaldson--Futaki invariant to the adjoint foliated setting. We then prove that adjoint Calabi--Yau foliated structures are K-semistable, and klt ones are K-stable, while log canonical adjoint general type foliated structures are K-stable with respect to the canonical polarisation. We also show that K-semistable adjoint Fano foliated structures are klt. In particular, their ambient varieties are potentially klt and of Fano type.
\end{abstract}

\maketitle

% \setcounter{tocdepth}{1}
% {\hypersetup{hidelinks}
% \tableofcontents}

\section{Introduction}

In the last two decades, the theory of K-stability has become a central bridge between differential and birational geometry. Its original motivation comes from the existence problem for K{\"a}hler--Einstein metrics on Fano varieties, culminating in the solution of the Yau--Tian--Donaldson conjecture for Fano manifolds and
varieties \cite{CDS2013}, and in the construction of K-moduli spaces of Fano varieties (see \cite{Xu2025} for an excellent survey). Although these applications focus on the Fano case, K-stability is defined for arbitrary polarised varieties. A series of key results by Odaka \cite{Odaka11, Odaka12, Odaka13} show that K-stability imposes strong singularity conditions. K-semistable varieties have semi-log canonical singularities, while K-semistable Fano varieties are klt. Furthermore, semi-log canonical Calabi--Yau and canonically polarised varieties are always K-semistable and K-stable respectively. Historically, these results have provided key links between the differential and birational geometric approaches to K-stability that has culminated with the K-moduli construction mentioned above.

The purpose of this paper is to extend this circle of ideas to adjoint foliated structures, i.e. triples $(X,\F,t)$ where $X$ is an ambient variety and $\F$ is an algebraically integrable foliation with adjoint canonical divisor $K^{[t]}_{X,\F} = (1-t)K_X+tK_\F$. A notion of K-stability for adjoint foliated structures was defined in \cite{Pap26}, motivated by the substantial recent progress on the MMP of adjoint foliated structures \cite{CHLMSSX24, CHLMSSX25, CLSV26}. In particular, our first main result is as follows. 

\begin{theorem}[see Theorem \ref{thm:mixed-odaka-singularity-obstruction-normal}]\label{thm: intro main 1}
    Let $(X,\F,t)$ be a normal projective $\Q$-Gorenstein
adjoint foliated structure with $\F$ algebraically integrable, and $0<t<1$. Let $L$ be an ample $\Q$-Cartier divisor on $X$. If $(X,\F,L)$ is $t$-K-semistable, then $(X,\F,t)$ is log canonical.
\end{theorem}

In order to prove Theorem \ref{thm: intro main 1}, we extend Odaka's flag ideal method to the adjoint foliated setting. We introduce a mixed $S$-coefficient associated to a flag ideal and show that its negativity forces negativity of the mixed Donaldson--Futaki invariant. If $(X,\F,t)$ is not log canonical, we prove that a relative log canonical modification for adjoint foliated structures exists, which allows us to extract divisors with negative mixed discrepancy. A suitably chosen flag ideal then produces an $\F$-compatible test configuration whose mixed $S$-coefficient is negative. This gives a destabilising test configuration, proving the contrapositive.

We then shift our focus to the three special classes of adjoint Calabi--Yau, general type and Fano foliated structures. In our second main result, we show that log canonical adjoint Calabi--Yau and general type foliated structures are K-semistable. We also show that K-semistable adjoint Fano foliated structures are klt. 

\begin{theorem}[See Theorems \ref{thm:mixed-odaka-CY}, \ref{thm:mixed-odaka-general-type} and \ref{thm:mixed-odaka-Fano}]\label{thm: main intro thm 2}
    Let $(X,\F,t)$ be a normal projective $\Q$-Gorenstein log canonical adjoint foliated structure for $0<t<1$. Then, 
    \begin{enumerate}
        \item if $(X,\F,t)$ is adjoint foliated Calabi--Yau then $(X,\F;L)$ is $t$-K-semistable for any ample polarisation $L$. Furthermore, if $(X,\F,t)$ is klt then $(X,\F;L)$ is $t$-K-stable for any ample polarisation $L$;
        \item if $(X,\F,t)$ is adjoint foliated general type then $(X,\F;L)$ is $t$-K-stable for $L = K_{X,\F}^{[t]}$;
        \item if $(X,\F,t)$ is adjoint foliated Fano which is $t$-K-semistable for $L = -K_{X,\F}^{[t]}$, then $(X,\F,t)$ is klt. In particular, $X$ is potentially klt, and of Fano type.
    \end{enumerate}
\end{theorem}

The proof of parts 1 and 2 of Theorem \ref{thm: main intro thm 2} do not use the S coefficients as Theorem \ref{thm: intro main 1}, but instead estimate the components comprising the mixed Donaldson--Futaki invariant, extending \cite{Odaka11}. In particular, we take an an arbitrary normal $\F$-compatible test configuration and then pass to an equivariant common model dominating both the compactified test configuration and the product $X\times \PP^1$. On this model, the mixed Donaldson--Futaki invariant decomposes into a canonical term and a discrepancy term. The log canonicity assumption implies that the discrepancy term is non-negative. In the Calabi--Yau case the canonical term vanishes, while in the general type case, with polarisation proportional to $K^{[t]}_{X,\F}$, it is non-negative and is positive for non-trivial test configurations. Combining these, we obtain the required positivity that proves the first two parts of Theirem \ref{thm: main intro thm 2}.

For the adjoint Fano case, Theorem \ref{thm: intro main 1} already gives us log canonicity. As such we assume that $(X,\F,t)$ is not klt and use a qdlt modification extracting a divisor of mixed discrepancy zero. For the resulting test configuration, the discrepancy term vanishes while the canonical term is negative, contradicting $t$-K-semistability. Thus $t$-K-semistable adjoint Fano foliated structures
are klt. The final consequences follow from \cite[Theorem 1.10(1)]{CHLMSSX24}. In particular, this justifies the technical choice in \cite[\S 5]{Pap26}, where all ambient varieties are assumed to be potentially klt.

\subsection*{Structure of the paper}
The paper is organised as follows. In Section \ref{sec:preliminaries} we recall the necessary background on adjoint foliated structures, mixed log discrepancies, and $t$-K-stability. In Section \ref{sec: S-coeff_full} we introduce the mixed $S$-coefficient and prove the flag-ideal negativity criterion. In Section \ref{sec: Odaka sing cond 1} we prove the relative lc modification for adjoint foliated structures and we apply the criterion to prove Theorem \ref{thm: intro main 1}. In Section \ref{sec: odaka sing cond 2} we prove the Calabi--Yau, general type, and Fano cases of Theorem \ref{thm: main intro thm 2}.

\subsection*{Acknowledgments}
I would like to thank Caucher Birkar, Federico Bongiorno, Paolo Cascini, Ruadhai Dervan and Jihao Liu for many helpful conversations and valuable comments. I am supported by  Beijing Natural Science Foundation Project IS25037 and a Shuimu Scholar Programme Scholarship at Tsinghua University.

\section{Preliminaries}\label{sec:preliminaries}
\subsection{Foliations and adjoint foliated structures}

In this subsection we recall the basic definitions concerning foliations and adjoint foliated structures. We follow the standard conventions in the birational geometry of foliations; for general background on foliations we refer for instance to \cite{Brunella15,AD13, ACSS21, CS25}, and for adjoint foliated structures to \cite{CHLMSSX24,CHLMSSX25, CLSV26}.

\begin{definition}\label{def: foliation}
Let $X$ be a normal variety. A \emph{foliation} on $X$ is a saturated coherent subsheaf $T_\F\subseteq T_X$ which is closed under the Lie bracket. Its rank is the generic rank of $T_\F$, and its codimension is $\codim(\F):=\dim X-\rk({\F})$. Equivalently, we can regard $\F$ as an integrable distribution on $X$.
\end{definition}

Since $T_{\F}$ is saturated in the reflexive sheaf $T_X$, it is reflexive. In particular, $\det(T_{\F})$ is a rank one reflexive sheaf, hence corresponds to a Weil divisor class on $X$.

\begin{definition}
Let $X$ be a normal variety and let $\F$ be a foliation on $X$. We say that $\F$ is \emph{algebraically integrable} if the leaf of $\F$ through a very general point of $X$ is Zariski open in an algebraic subvariety of $X$. Equivalently, the closure of the leaf through a very general point is an algebraic subvariety.
\end{definition}

Equivalently, $\F$ is algebraically integrable if its general leaves are algebraic. In this case, after possibly replacing $X$ by a birational model, one may often view $\F$ as induced by a rational map
\[
\begin{tikzcd}
f\colon X \arrow[r, dashed] & Y
\end{tikzcd}
\]
whose general fibres are tangent to $\F$.

\begin{definition}\label{def: canonical class of foliation}
Let $X$ be a normal variety and let $\F$ be a foliation on $X$. The \emph{canonical class} of $\F$ is the Weil divisor class $K_{\F}$ determined by
\[
\OO_X(-K_{\F})\simeq \det(T_{\F}),
\]
or equivalently
\[
\OO_X(K_{\F})\simeq \det(T_{\F})^{\vee}.
\]
When $X$ is smooth, if $N_{\F}^{\vee}:=(T_X/T_{\F})^{\vee}$ denotes the conormal sheaf of $\F$, then
\[
K_X = K_{\F} + \det(N_{\F}^{\vee}).
\]
\end{definition}

\begin{definition}
A \emph{foliated pair} consists of a normal variety $X$ together with a foliation $\F$ such that $K_{\F}$ is $\QQ$-Cartier. Furthermore, we call the triple $(X,\F,L)$ a \emph{foliated polarised variety} if $(X,\F)$ is a foliated pair, and $L$ is an ample divisor on $X$.
\end{definition}

We now recall the main object of interest in this paper.

\begin{definition}\label{def: adjoint foliated structure}
An \emph{adjoint foliated structure} is a triple $(X,\F,t)$, where $X$ is a normal variety, $\F$ is a foliation on $X$, and $t\in[0,1]$. To such a triple we associate the $\QQ$-divisor
\[
K_{X,\F}^{[t]}:= tK_{\F}+(1-t)K_X
\]
called the \emph{mixed canonical divisor}. We say that $(X,\F,t)$ is $\QQ$-Gorenstein if $K_{X,\F}^{[t]}$ is $\QQ$-Cartier.
\end{definition}

For our purposes we will always implicitly assume that $\F$ is algebraically integrable, and a foliation will always mean an algebraically integrable foliation.

\begin{definition}
    We say that the adjoint foliated structure $(X,\F,t)$ is an \emph{adjoint general type foliated structure} if $K_{X,\F}^{[t]}$ is ample. Similarly, we say that the adjoint foliated structure $(X,\F,t)$ is an \emph{adjoint Fano foliated structure} if $-K_{X,\F}^{[t]}$ is ample. Furthermore, we say that the adjoint foliated structure $(X,\F,t)$ is an \emph{adjoint Calabi--Yau foliated structure} if $K_{X,\F}^{[t]}$ is (numerically) trivial.
\end{definition}

\subsection{Foliated log discrepancy}

Let $\mu:Y\to X$ be a birational morphism extracting a prime divisor $E$; denote by $\cF_Y$ the induced foliation on $Y$. We say that $E$ is \emph{$\cF$-invariant} if the foliation is tangent to $E$ generically (equivalently, $\cF_Y\subset T_Y(-\log E)$ generically); otherwise $E$ is \emph{transverse}. The foliated canonical class satisfies
$$
K_{\cF_Y}\ =\ \mu^*K_{\cF}+\sum_E a(E,\cF)\,E.
$$
We define
$$
\varepsilon(E)\ :=\
\begin{cases}
1, & E \text{ transverse},\\
0, & E \text{ $\cF$-invariant},
\end{cases}
\qquad
{\ A_{X,\cF}(E)\ :=\ \varepsilon(E)+a(E,\cF)\ }.
$$

For an adjoint foliated structure $(X,\F,t)$ we define the \emph{mixed log discrepancy},
\[
A^{[t]}_{X,\F}(E)
=
t\bigl(a(E,\F)+\varepsilon(E)\bigr)+(1-t)\bigl(a(E,X)+1\bigr).
\]

\begin{definition}
Let $(X,\F,t)$ be an adjoint foliated structure with $t\in[0,1]$. We say that $(X,\F,t)$ is
\begin{enumerate}
    \item \emph{log canonical} if $A^{[t]}_{X,\F}(E)\ge 0$ for every prime divisor $E$ over $X$;
\item \emph{klt} if $A^{[t]}_{X,\F}(E)> 0$ for every prime divisor $E$ over $X$;
\item \emph{$\epsilon$-log canonical} if $A^{[t]}_{X,\F}(E)\ge \epsilon$ for every prime divisor $E$ over $X$;
\item \emph{$\epsilon$-klt} if $A^{[t]}_{X,\F}(E)> \epsilon$ for every prime divisor $E$ over $X$.
\end{enumerate}
\end{definition}

When $t=0$, this definition recovers the usual singularity condition for $X$. When $t=1$, this definition recovers the foliated singularity condition. We will say that $X$ is \emph{potentially klt} if there exists effective divisor $\Delta$, such that $(X,\Delta)$ is klt.

\subsection{K-stability for adjoint foliated structures}

In this section we recall the notions of foliated test configurations and of K-stability for adjoint foliated structures that appeared in \cite{Pap26}.

Let $X$ be a normal projective variety of dimension $n$ over $\mathbb{C}$ and $L$ an ample $\mathbb{Q}$-line bundle on $X$. Let $\mathcal{F}\subset T_X$ be an algebraically integrable foliation such that the canonical divisor $K_{\mathcal{F}}$ is $\mathbb{Q}$-Cartier. 
% For $t\in[0,1]$ we define the \emph{mixed canonical divisor}
% \[
% K^{[t]}_{X,\F}:=(1-t)K_X+tK_{\F}.
% \]

\begin{definition}\label{def: foliated t.c.}
A \emph{foliated test configuration} for the polarised foliated variety $(X,\F,L)$ consists of a triple
\[
(\pi:\X\to \A^1,\ \F_{\X},\ \L)
\]
such that:

\begin{enumerate}
\item $(\X,\L)$ is a normal test configuration for $(X,L)$, i.e.\ $\pi:\X\to\A^1$ is a flat projective morphism endowed with a $\G_m$-action lifting the standard action on $\A^1$, $\L$ is a relatively ample $\Q$-line bundle, and
\[
(\X,\L)|_{\pi^{-1}(\A^1\setminus\{0\})}
\simeq (X,L)\times (\A^1\setminus\{0\})
\]
$\G_m$-equivariantly;

\item $\F_{\X}\subset T_{\X/\A^1}$ is a
$\G_m$-equivariant saturated integrable subsheaf;

\item over $\A^1\setminus\{0\}$ the foliation coincides with the
product foliation:
$$(\X,\F_{\X})|_{\pi^{-1}(\A^1\setminus\{0\})}
\simeq (X,\F)\times (\A^1\setminus\{0\});$$

\item $\F_{\X}$ is algebraically integrable.
\end{enumerate}

Such a test configuration will be called \emph{$\F$-compatible}.
\end{definition}

Given an arbitrary $\F$-compatible test configuration $\pi:(\X,\F_{\X},\L)\rightarrow \A^1$, we take $\bar{\pi}:(\bar{\X},\F_{\bar{\X}},\bar{\L})\rightarrow \PP^1$ for the natural compactification and we set $V:=L^n$ and 
\[\mu(X,\F,L) = \frac{-K^{[t]}_{X,\F}\cdot L^{n-1}}{L^n}\]
for the \emph{foliated slope}. Furthermore, we define the \emph{relative mixed canonical divisor} by
\[
K^{[t]}_{\bar{\X}/\PP^1}
:=
(1-t)K_{\bar{\X}/\PP^1}+tK_{\F_{\bar{\X}}}.
\]

\begin{definition}\label{def:DFwt-corrected}
Let $(\X,\ \F_{\X},\ \L)$
be a normal $\F$-compatible test configuration for $(X,\F,L)$.
We fix $t\in (0,1)$. The \emph{$t$-foliated Donaldson--Futaki invariant} is defined by
\[
\DF^{[t]}(\X,\F_{\X},\L) := \frac{1}{V}\left(\frac{n}{n+1}\,\mu(X,\F,L)\bar{\L}^{n+1}+ K^{[t]}_{\bar{\X}/\PP^1}\cdot \bar{\L}^n\right).
\]
\end{definition}

We will also call this invariant the ``mixed Donaldson--Futaki'' invariant for brevity. Note that the notions of product, trivial and special test configurations extend to those of $\F$-compatible product, trivial and special test configurations in the natural way.

\begin{definition}\label{def: k-stability definition}
Let $(X,\F,L)$ be a polarised foliated variety.

\begin{enumerate}
\item $(X,\F,L)$ is \emph{$t$-K-semistable}
if
\[
\DF^{[t]}(\X,\F_{\X},\L)\ge0
\]
for all normal $\F$-compatible test configurations.

\item $(X,\F,L)$ is \emph{$t$-K-stable}
if
\[
\DF^{[t]}(\X,\F_{\X},\L)>0
\]
for all non-trivial $\F$-compatible test configurations.

\item $(X,\F,L)$ is \emph{uniformly $t$-K-stable}
if there exists $\delta>0$ such that
\[
\DF^{[t]}(\X,\F_{\X},\L)
\ge
\delta\,J^{NA}(\X,\L)
\]
for all normal $\F$-compatible test configurations.
\item $(X,\F,L)$ is \emph{$t$-K-polystable} if it is $t$-K-semistable and \[
\DF^{[t]}(\X,\F_{\X},\L)= 0 \text{ if and only if }(\X,\F_{\X}, \L) \text{ is $\F$-compatible of product type.}
\] 
\end{enumerate}
\end{definition}
Here 
\[
J^{NA}(\X,\L)
:=
\frac{1}{V}
\left(
\L\cdot (\mu^*L)^n
-
\frac{1}{n+1}\L^{n+1}
\right)
\]
is the non-Archimedean $J$-functional of the test configuration. Note, that we will interchangably say that either $(X,\F,L)$ is $t$-K-semistable, or that the adjoint foliated structure $(X,\F,t)$ is $t$-K-semistable, especially if the polarisation is implicit.

\section{\texorpdfstring{The mixed $S$-coefficient and Odaka's negativity criterion}{The mixed S-coefficient and a negativity criterion}}\label{sec: S-coeff_full}

In this section we will define the mixed $S$-coefficient extending \cite[\S 3]{Odaka13}, and we will show that negativity of the $S$-coefficient implies K-instability. We will achieve this by studying the behaviour of the blow ups of flag ideals that give rise to $\F$-compatible test configurations, and by studying their associated mixed Donaldson--Futaki invariants.

\subsection{Mixed log discrepancies and the relative canonical term}\label{sec: mixed_log_dic}

Throughout this section $X$ is a normal projective variety of dimension $n$, $\F\subset T_X$ is an algebraically integrable foliation with $K_{\F}$ $\Q$-Cartier, and $t\in[0,1]$. We begin with the local coefficient computation which relates mixed log discrepancies on $X$ to relative mixed canonical divisors on birational models of $X\times\A^1$.

Let $p_1:X\times\A^1\to X$ be the first projection, and let $s$ denote the coordinate on $\A^1$. We write $p_1^{-1}\F$ for the product
foliation on $X\times\A^1$, regarded as a relative foliation over
$\A^1$.

\begin{definition}
Let $\Pi:\Y\to X\times\A^1$ be a normal birational model, and let $\F_{\Y}$ be the saturated birational transform of the product foliation. We define
\[
K^{[t]}_{\Y/(X\times\A^1)} := (1-t)K_{\Y/(X\times\A^1)}+ t\left(K_{\F_{\Y}}-\Pi^*p_1^*K_{\F}\right).
\]
\end{definition}

Let $\Pi:\Y\to X\times\A^1$ be a normal birational model carrying the saturated birational transform $\F_{\Y}$ of the product foliation. Let $G\subset \Y$ be a prime divisor whose centre on $X\times\A^1$ is contained in $X\times\{0\}$. Throughout we assume that $w:=\ord_G$ is $\G_m$-invariant and let $v:=w|_{K(X)}$, $r:=w(s)$. The following Lemma allows us to compare the coefficient of $K^{[t]}_{\Y/(X\times\A^1)}$ with the mixed log discrepancy. A similar version also appeared in \cite[Proposition 4.3]{Pap26}, with more restrictive assumptions. We include an expansion here with a proof for the reader's convenience.

\begin{lemma}
\label{lem:mixed-gauss-rees-coefficient}
We have 
\[
\operatorname{coeff}_G K^{[t]}_{\Y/(X\times\A^1)} = A^{[t]}_{X,\F}(v)+(1-t)(r-1).
\]
In particular, if $r=1$, then
\[
\operatorname{coeff}_G K^{[t]}_{\Y/(X\times\A^1)}=A^{[t]}_{X,\F}(v).
\]
\end{lemma}
\begin{proof}
The statement is local at the generic point of $G$. Since $w$ is invariant under the $\G_m$-action and the $f_is^i$ are weight vectors, $w$ is determined by its values on homogeneous components. Hence for a finite sum of homogeneous terms we obtain the formula
\[
w\left(\sum_i f_i s^i\right)
=
\min_i\{v(f_i)+ir\}.
\]
Equivalently, $w$ is the homogeneous extension of $v:=w|_{K(X)}$ with $w(s)=r$.

We choose a birational model $\mu:Z\to X$ on which $v=c\ord_P$ for a prime divisor $P\subset Z$. Replacing $Z$ by a higher model if necessary, we may assume that $Z$ is smooth at the generic point of $P$, that $P=\{z=0\}$ locally, and that the birational transform $\F_Z$ is regular at this generic point. This is possible since $T_{\F_Z}$ is saturated in $T_Z$, and on a smooth model, a saturated subsheaf of a locally free sheaf is locally free in codimension one, hence $T_{\F_Z}$ is locally free at the generic point of the divisor $P$. After restricting to that generic point, the foliation is regular.

We may also replace $\Y$ by a higher model dominating $Z\times\A^1$, since the coefficient of $G$ in the relative divisor is computed at the generic point of $G$. Thus, locally, we may compute the required coefficients on a model
\[
\rho:\Y\longrightarrow Z\times \A^1
\]
which extracts $G$. We have $w(z)=c$ and $ w(s)=r$.

The ordinary product discrepancy computation gives
\[
\operatorname{coeff}_G K_{\Y/(X\times\A^1)} = A_X(v)+r-1
\]
since the contribution from the $X$-direction is $A_X(v)$, while the extension with weight $r=w(s)$ contributes $r-1$ from the base parameter.

We now compute the foliated contribution. We write
\[
K_{\F_Z} = \mu^*K_{\F}+a(P,\F)P+\cdots .
\]
Pulling back to $Z\times\A^1$, the divisor $P\times\A^1$ appears in $K_{p_1^{-1}\F_Z}-p_1^*K_{\F}$ with coefficient $a(P,\F)$. Since $w(z)=c$, this contribution gives $c\,a(P,\F)$ to the coefficient along $G$.

Let $q=\operatorname{rk}\F$. Since the computation is local at the generic point of $P$, we may choose a local frame $\delta_1,\ldots,\delta_q$ for $T_{\F_Z}$. We regard these as relative vector fields on $Z\times\A^1$ by requiring $\delta_i(s)=0$ for every $i$. The saturated transform of the product foliation on
$\Y$ is obtained by saturating the pullbacks of these relative vector
fields inside $T_{\Y/\A^1}$.

First suppose that $P$ is $\F_Z$-invariant. Then $\delta_i(z)\in (z)$ for every $i$. Equivalently, $\delta_i(z)/z$ is regular at the generic point of $P$. Let $m=z^a s^b$ be a local toroidal monomial. Since $\delta_i(s)=0$, we have
\[
\delta_i(m)
=
m\left(a\frac{\delta_i(z)}{z}
      +b\frac{\delta_i(s)}{s}\right)
=
a\,m\frac{\delta_i(z)}{z}.
\]
Thus $w(\delta_i(m))\ge w(m)$, and hence the pullback of each $\delta_i$ has no pole along $G$. After saturation, no additional factor of a local equation of $G$ is introduced. Therefore
\[
\operatorname{coeff}_G
\left(
K_{\F_{\Y}}-\rho^*K_{p_1^{-1}\F_Z}
\right)
=
0.
\]
This is $c\,\varepsilon_{\F}(P)$ in the invariant case, since
$\varepsilon_{\F}(P)=0$.

Now let $P$ be $\F_Z$-transverse. {\'E}tale-locally, or after passing to the corresponding toroidal local model, the valuation $w$ is represented by coordinates $z=x^c$, $s=x^r u$ with $G=\{x=0\}$. Let $D$ be the rational relative lift of $\delta_1$, so that $D(z)=1$, $D(s)=0$. Writing $D=A\partial_x+B\partial_u$, these two equations give $c x^{c-1}A=1$ and $r x^{r-1}uA+x^rB=0$. Hence $A=\frac{1}{c x^{c-1}}$ and $B=-\frac{r u}{c x^c}$. Thus $D$ has pole order exactly $c$ along $G$. Multiplying by $z=x^c$ removes the pole:
\[
        zD
        =
        \frac{x}{c}\partial_x-\frac{r u}{c}\partial_u,
\]
which is regular and nonzero at the generic point of $G$. Therefore, after saturation, the transverse generator contributes exactly one factor of valuation $c=w(z)$, and no factor involving $r$. Consequently
\[
        \det T_{\F_Y}
        =
        \rho^*\det T_{p_1^{-1}\F_Z}\otimes \mathcal O_Y(-cG)
\]
near the generic point of $G$, or equivalently,
\[
K_{\F_{\Y}}
=
\rho^*K_{p_1^{-1}\F_Z}+cG+\cdots .
\]
Hence
\[
\operatorname{coeff}_G
\left(
K_{\F_{\Y}}-\rho^*K_{p_1^{-1}\F_Z}
\right)
=
c.
\]
This is $c\,\varepsilon_{\F}(P)$ in the transverse case, since
$\varepsilon_{\F}(P)=1$.

Combining the invariant and transverse cases, we obtain
\[
\operatorname{coeff}_G
\left(
K_{\F_{\Y}}-\rho^*K_{p_1^{-1}\F_Z}
\right)
=
c\,\varepsilon_{\F}(P).
\]
Adding the contribution already coming from
$K_{\F_Z}-\mu^*K_{\F}$, we get
\[
\operatorname{coeff}_G
\left(
K_{\F_{\Y}}-\Pi^*p_1^*K_{\F}
\right)
=
c\,a(P,\F)+c\,\varepsilon_{\F}(P).
\]
Since $v=c\ord_P$, the right-hand side is precisely $A_{X,\F}(v)$. Therefore
\[
\operatorname{coeff}_G
\left(
K_{\F_{\Y}}-\Pi^*p_1^*K_{\F}
\right)
=
A_{X,\F}(v).
\]

Now, combining this expression with the ambient discrepancy computation gives
\[
\begin{aligned}
\operatorname{coeff}_G K^{[t]}_{\Y/(X\times\A^1)} &= (1-t)\operatorname{coeff}_G K_{\Y/(X\times\A^1)} + t\operatorname{coeff}_G \left( K_{\F_{\Y}}-\Pi^*p_1^*K_{\F}\right) \\
&=(1-t)(A_X(v)+r-1)+tA_{X,\F}(v) \\
&=A^{[t]}_{X,\F}(v)+(1-t)(r-1).
\end{aligned}
\]
\end{proof}

Using the above Lemma, we obtain natural effectivity of the divisor $K^{[t]}_{\Y/(X\times\A^1)}$, which will be used in mixed Donaldson--Futaki invariant computations.

\begin{corollary}
\label{cor:mixed-relative-effective}
Assume that $(X,\F,t)$ is log canonical. Let $\Pi:\Y\to X\times\A^1$ be a normal $\G_m$-equivariant birational model which is an isomorphism over $X\times(\A^1\setminus\{0\})$, and let $\F_{\Y}$ be the saturated birational transform of the product foliation. Then $K^{[t]}_{\Y/(X\times\A^1)}$ is effective.
\end{corollary}

\begin{proof}
Let $G$ be a prime divisor on $\Y$ exceptional over $X\times\A^1$. Since $\Pi$ is an isomorphism away from the central fibre, the centre of $G$ is contained in $X\times\{0\}$. Since the construction is $\G_m$-equivariant and $\G_m$ is connected, $G$ is $\G_m$-invariant. Moreover, since $\Pi$ is an isomorphism at the generic point of every non-$\Pi$-exceptional divisor and $\F_Y$ is the saturated birational transform of the product foliation, the divisor $K^{[t]}_{\Y/(X\times \A^1)}$ is $\Pi$-exceptional. Thus it suffices to check the coefficients along $\Pi$-exceptional prime divisors.

Let, as before, $v=\ord_G|_{K(X)}$ and $r=\ord_G(s)$. By Lemma \ref{lem:mixed-gauss-rees-coefficient},
\[
\operatorname{coeff}_G K^{[t]}_{\Y/(X\times\A^1)}=A^{[t]}_{X,\F}(v)+(1-t)(r-1).
\]
Since $(X,\F,t)$ is log canonical, $A^{[t]}_{X,\F}(v)\ge0$, and $1-t>0$. Moreover $r=\ord_G(s)\ge1$. Therefore the coefficient of $G$ is nonnegative. This proves the required effectivity.
\end{proof}

\subsection{\texorpdfstring{The mixed $S$-coefficient}{The mixed S-coefficient}}

We next compute the expansion fo the mixed Donaldson--Futaki invariant in normalised blow ups of flag-ideals. Let us first recall the definition of flag ideals.

\begin{definition}
    A coherent ideal sheaf $\J$ of $X \times \A^1$ is called a \emph{flag ideal} if $\J = I_0+I_1k+\dots +I_{N-1}k^{N-1}+(k^N)$, where $I_0 \subseteq I_1 \subseteq \dots\subseteq I_{N-1} \subseteq \cO_X$ is the sequence of coherent ideal sheaves . It is equivalent to that the corresponding subscheme is supported on the central fiber $X\times \{0\}$ and is $\G_m$-invariant under the natural action of $\G_m$ on $X \times \A^1$.
\end{definition}

Let $\J\subset \cO_{X\times\A^1}$ be a $\G_m$-invariant flag ideal, and let
\[
\Pi:\B:=\operatorname{Bl}_{\J}(X\times\A^1)^\nu
\longrightarrow X\times\A^1
\]
be the normalised blow-up, with exceptional divisor $E$. We will write $\cO_{\B}(-E):= \J\cdot\cO_{\B}$ and $\M:=\Pi^*p_1^*L$. For $r\gg0$, we define $\L_r:=r\M-E$ and compactify over $\PP^1$. By abuse (and ease) of notation we will use the same notation for the compactified objects. Let $Z:=\operatorname{Supp}\left(\cO_{X\times\A^1}/\J\right)$ where $s:=\dim Z$. 

\begin{definition}
The \emph{mixed $S$-coefficient} of $\J$ with respect to
$(X,\F,L,t)$ is
\[
S^{[t]}_{(X,\F),L}(\J) := \M^s\cdot(-E)^{n-s} \cdot K^{[t]}_{\B/(X\times\PP^1)}.
\]
\end{definition}

We now show how the $S$-coefficient characterises $t$-K-stability.

\begin{lemma}
\label{lem:dimension-vanishing-flag}
Let $\Pi:\B\to X\times \PP^1$ be the compactified normalised blow-up of a flag ideal $\J$, and let $\M:=\Pi^*p_1^*L$. Let $T$ be a cycle on $\B$ whose support maps into a closed subset $Z\subset X\times\{0\}$ of dimension at most $s$. Then $M^j\cdot T=0$ for all $j>s$. Moreover, for every $\Q$-Cartier divisor $D$ on $X$, $\Pi^*p_1^*D\cdot M^j\cdot T=0$ for all $j \geq s$.
\end{lemma}

\begin{proof}
Since $Z\subset X\times\{0\}$, its image under $p_1$ has dimension at most $s$. The divisor $\M$ is pulled back from $X$. Thus, after replacing $L$ by a sufficiently divisible multiple, $j>s$ general members of $|L|$ have empty intersection with $p_1(Z)$. Pulling them back to $\B$ gives $\M^j\cdot T=0$. 

The second assertion follows from the same dimension count. The product $\Pi^*p_1^*D\cdot \M^j$ consists of $j+1$ divisor classes pulled back from $X$. If $j\ge s$, then $j+1>s$, so their intersection with a cycle mapping to $p_1(Z)$ is zero.
\end{proof}

We will denote by $\F_\B$ the (saturated) birational transform of the product foliation. Note that by \cite[Lemma 7.12]{Pap26}, $(\B,\F_\B, \L_r)$ is a (semi-) ample test configuration.

\begin{proposition}
\label{prop:mixed-S-leading-term}
There exists an assymptotic expansion
\[
\DF^{[t]}(\B,\F_{\B},\L_r) = \frac{\binom ns}{L^n}
S^{[t]}_{(X,\F),L}(\J)r^{s-n} + O(r^{s-n-1}).
\]
In particular, if $S^{[t]}_{(X,\F),L}(\J)<0$ then $(X,\F,L)$ is not $t$-K-semistable.
\end{proposition}
\begin{proof}
We fix $V:=L^n$. The general fibre of $(\B,\L_r)\to \A^1$ is $(X,rL)$, hence its volume is $r^nV$. Moreover we have $\mu(X,\F,rL) = \frac{1}{r}\mu^{[t]}_L$ and $\mu^{[t]}_L:= \frac{-K^{[t]}_{X,\F}\cdot L^{n-1}}{L^n}$. Using the fact that
\[
K^{[t]}_{\B/\PP^1}= \Pi^*p_1^*K^{[t]}_{X,\F}+ K^{[t]}_{\B/(X\times\PP^1)},
\]
we can decomose the mixed Donaldson--Futaki invariant as $\DF^{[t]}(\B,\F_{\B},\L_r)=\operatorname{Can}^{[t]}_r+\operatorname{Disc}^{[t]}_r$, where
\[
\operatorname{Disc}^{[t]}_r= \frac{1}{r^nV}K^{[t]}_{\B/(X\times\PP^1)} \cdot (r\M-E)^n
\]
and
\[
\operatorname{Can}^{[t]}_r=\frac{1}{r^nV}\left(\frac{n}{n+1}\frac{\mu^{[t]}_L}{r}(r\M-E)^{n+1}+\Pi^*p_1^*K^{[t]}_{X,\F}\cdot (r\M-E)^n\right).
\]

We first compute the discrepancy term. Expanding,
\[
(r\M-E)^n=\sum_{j=0}^n \binom nj r^j\M^j(-E)^{n-j}.
\]
Since $\Pi$ is an isomorphism over $(X\times \PP^1)\setminus Z$, and since $\F_{\B}$ is the saturated birational transform of the product foliation, both the ordinary relative canonical divisor and the foliated relative canonical divisor are supported on the $\Pi$-exceptional locus. Hence $K^{[t]}_{\B/(X\times \PP^1)}$ is $\Pi$-exceptional and its support maps into $Z=\operatorname{Supp}(\cO_{X\times\A^1}/\J)$. Since $\dim Z=s$, Lemma \ref{lem:dimension-vanishing-flag} implies that all terms with $j>s$ vanish. Hence
\[
\operatorname{Disc}^{[t]}_r=\frac{1}{r^nV}\left( \binom ns r^s\M^s(-E)^{n-s}\cdot K^{[t]}_{\B/(X\times\PP^1)} + O(r^{s-1}) \right).
\]
By the definition of the mixed $S$-coefficient, we obtain
\[
\operatorname{Disc}^{[t]}_r = \frac{\binom ns}{V} S^{[t]}_{(X,\F),L}(\J)r^{s-n} + O(r^{s-n-1}).
\]

It remains to show that the canonical part is $O(r^{s-n-1})$. Consider first $\frac{1}{r}(r\M-E)^{n+1}$. Expanding, we get
\[
\frac{1}{r}(r\M-E)^{n+1}= \sum_{j=0}^{n+1}\binom{n+1}{j}r^{j-1}\M^j(-E)^{n+1-j}.
\]
The term $j=n+1$ vanishes because $\M^{n+1}=0$, since $\M$ is pulled back from the $n$-dimensional variety $X$. Every remaining term contains at least one factor of $E$, hence is supported over $Z$. By Lemma \ref{lem:dimension-vanishing-flag}, all terms with $j>s$ vanish. Thus the largest possible power is $r^{s-1}$. After division by $r^nV$, this contributes $O(r^{s-n-1})$.

Now consider the second canonical term
\[
\Pi^*p_1^*K^{[t]}_{X,\F}\cdot (r\M-E)^n.
\]
Expanding, we get
\[
\Pi^*p_1^*K^{[t]}_{X,\F}\cdot (r\M-E)^n=\sum_{j=0}^n\binom nj r^j \Pi^*p_1^*K^{[t]}_{X,\F}\cdot \M^j(-E)^{n-j}.
\]
The term $j=n$ vanishes because it is the pullback of an$(n+1)$-fold intersection on $X$:
\[
\Pi^*p_1^*K^{[t]}_{X,\F}\cdot \M^n = \Pi^*p_1^*\left(K^{[t]}_{X,\F}\cdot L^n\right) = 0.
\]
For $j<n$, the term contains at least one factor of $E$, hence is supported over $Z$. If $j>s$, it vanishes by Lemma \ref{lem:dimension-vanishing-flag}. If $j=s$, it also vanishes, because it contains $s+1$ divisor classes pulled back from $X$, namely $K^{[t]}_{X,\F}$ and $s$ copies of $L$, intersected with a cycle mapping to a subset of $X$ of dimension at most $s$. Therefore the largest possible nonzero power is $r^{s-1}$. After division by $r^nV$, this term is also $O(r^{s-n-1})$.

Consequently $\operatorname{Can}^{[t]}_r=O(r^{s-n-1})$. Combining this with the expansion of $\operatorname{Disc}^{[t]}_r$ gives
\[
\DF^{[t]}(\B,\F_{\B},\L_r) = \frac{\binom ns}{L^n} S^{[t]}_{(X,\F),L}(\J)r^{s-n} + O(r^{s-n-1}).
\]
In particular, if $S^{[t]}_{(X,\F),L}(\J)<0$, then the mixed Donaldson--Futaki invariant is negative for all sufficiently large $r$. Hence $(X,\F,L)$ is not $t$-K-semistable.
\end{proof}

The following lemma will allow us to estimate the intersection numbers that appear in Proposition \ref{prop:mixed-S-leading-term}.

\begin{lemma}
\label{lem:relative-intersection-sign}
Let $G$ be an exceptional prime divisor on $\B$, and set
$Z_G:=\Pi(G)$. Then
\[
\M^s\cdot(-E)^{n-s}\cdot G=0
\]
if $\dim Z_G<s$, while
\[
\M^s\cdot(-E)^{n-s}\cdot G>0
\]
if $\dim Z_G=s$.
\end{lemma}

\begin{proof}
If $\dim Z_G<s$, then $s$ general divisors pulled back from $|mL|$ miss $Z_G$, so the intersection is zero.

Assume $\dim Z_G=s$. We then choose $s$ sufficiently general divisors $H_1,\ldots,H_s\in |mL|$. Their pullbacks meet $G$ in a nonzero effective cycle of dimension $n-s$, supported in finitely many fibres of $\Pi$. Since $-E$ is $\Pi$-ample on the flag-ideal blow-up, its restriction to those fibres is ample. Therefore the degree of $(-E)^{n-s}$ on this nonzero effective cycle is strictly positive.
\end{proof}

From now on we will write $K^{[t]}_{\B/(X\times\PP^1)} = \sum_i a_i^{[t]}G_i$ as a sum over exceptional prime divisors. The following proposition gives us a useful characterisation of $t$-K-instability in terms of the S-coefficient.

\begin{proposition}
\label{prop:mixed-odaka-negativity}
Assume that $a_i^{[t]}\le0$ for every $G_i$ with $\dim \Pi(G_i)=s$, and that at least one such coefficient is strictly negative. Then
\[
S^{[t]}_{(X,\F),L}(\J)<0.
\]
Consequently, $(X,\F,L)$ is not $t$-K-semistable.
\end{proposition}

\begin{proof}
By definition,
\[
S^{[t]}_{(X,\F),L}(\J)= \sum_i a_i^{[t]}\left(\M^s\cdot(-E)^{n-s}\cdot G_i\right).
\]
If $\dim\Pi(G_i)<s$, the corresponding intersection number is zero by Lemma \ref{lem:relative-intersection-sign}. If $\dim\Pi(G_i)=s$, then the corresponding intersection number is strictly positive. Hence
\[
S^{[t]}_{(X,\F),L}(\J) = \sum_{\dim\Pi(G_i)=s} a_i^{[t]} \left(\M^s\cdot(-E)^{n-s}\cdot G_i\right).
\]
All summands are nonpositive since $a_i^{[t]}\le0$, and at least one is strictly negative, by assumption. Therefore
\[
S^{[t]}_{(X,\F),L}(\J)<0.
\]
The final statement follows from Proposition \ref{prop:mixed-S-leading-term}.
\end{proof}

\section{Singularity condition on K-stability for adjoint foliated structures}\label{sec: Odaka sing cond 1}

We now prove Theorem \ref{thm: intro main 1}. We first record the following technical Lemma that will allow us to take relative log canonical modifications of adjoint foliated structures.

\begin{lemma}
\label{lem:relative-mixed-lc-modification}
Let $(X,\F,t)$ be a normal $\Q$-Gorenstein algebraically integrable adjoint foliated structure with $0<t<1$. Assume that $(X,\F,t)$ is not log canonical. Then there exists a projective birational morphism $\pi:Y\to X$ such that, writing $\F_Y$ for the induced foliation and $E_{\pi}=\sum_i E_i$ for the reduced $\pi$-exceptional divisor, the following hold.
\begin{enumerate}
\item The adjoint foliated structure
\[
\left(Y,\F_Y,\sum_i\bigl((1-t)+t\varepsilon_{\F}(E_i)\bigr)E_i,
0,t\right)
\]
is log canonical.

\item The divisor
\[
K^{[t]}_{Y,\F_Y}+\sum_i\bigl((1-t)+t\varepsilon_{\F}(E_i)\bigr)E_i
\]
is $\pi$-ample.

\item Every $\pi$-exceptional prime divisor $E_i$ satisfies $A^{[t]}_{X,\F}(E_i)<0$.
\end{enumerate}
\end{lemma}

\begin{proof}
We use the boundary notation for adjoint foliated structures as in \cite{CHLMSSX24,CHLMSSX25}. Thus, for $\cA=(X,\F,B,0,t)$ we write
\[
K_{\cA}:=K^{[t]}_{X,\F}+B.
\]

We fix $\cA=(X,\F,0,0,t)$, with $B=0$, so $K_{\cA}=K^{[t]}_{X,\F}$. We first apply the existence theorem for $\Q$-factorial qdlt modifications of algebraically integrable adjoint foliated structures with $t<1$, namely \cite[Theorem~1.9, equivalently Theorem~3.6]{CHLMSSX24}. This gives a projective birational morphism $h:X'\to X$ such that, if $\F'=h^{-1}\F$, and if
\[
B'=
\operatorname{Exc}(h)^{\rm ninv}+ (1-t)\operatorname{Exc}(h)^{\rm inv},
\]
where $\operatorname{Exc}(h)^{\rm ninv}$ and $\operatorname{Exc}(h)^{\rm inv}$ denote respectively the sums of the non-$\F'$-invariant and $\F'$-invariant $h$-exceptional prime divisors, then $\cA'=(X',\F',B',0,t)$ is $\Q$-factorial qdlt (and in particular it is lc). In particular, $X'$ is klt, and thus it is potentially klt. Moreover, the morphism $h$ extracts only non-klt places of the original adjoint foliated structure.

We now construct the relative log canonical model of $\cA'$ over $X$. Let $H$ be an $h$-ample $\R$-Cartier divisor on $X'$. For $0<\epsilon\ll 1$, the adjoint foliated structure $\cA'_\epsilon:=(X',\F',B',\epsilon H,t)$ is lc after choosing $H$ general in its numerical class, and $K_{\cA'}+\epsilon H$ is pseudo-effective over $X$. By \cite[Theorem 8.1]{CHLMSSX25}, any $(K_{\cA'}+\epsilon H)$-MMP over $X$ with scaling terminates with a good minimal model.

As $\epsilon$ varies in a sufficiently small rational polytope, \cite[Theorem 6.5]{CHLMSSX25} implies that only finitely many
good minimal models occur. Hence, after shrinking the interval
$0<\epsilon\ll1$, the good models stabilise. Let $Y$ be the resulting
stable model. Passing to the limit as $\epsilon\to0$, we obtain a good
minimal model of $K_{\cA'}$ over $X$. We now take its relative ample model, and denote it again by $Y$ (by abuse of notation), which gives us a projective birational morphism $\pi:Y\to X$. Let $\F_Y$ be the induced foliation on $Y$, and let $E_\pi=\sum_i E_i$ for the reduced $\pi$-exceptional divisor. Then, the pushforward adjoint foliated structure
\[
\cA_Y=\left(Y,\F_Y,\sum_i\bigl((1-t)+t\varepsilon_{\F}(E_i)\bigr)E_i,0,t\right)
\]
is lc, and $K_{\cA_Y}$ is $\pi$-ample. Equivalently,
\[
K^{[t]}_{Y,\F_Y} + \sum_i\bigl((1-t)+t\varepsilon_{\F}(E_i)\bigr)E_i
\]
is $\pi$-ample.

It remains to identify the discrepancies of the divisors which survive on $Y$. Let
\[
D_{\pi}^{[t]} := K^{[t]}_{Y,\F_Y} + \sum_i\bigl((1-t)+t\varepsilon_{\F}(E_i)\bigr)E_i - \pi^*K^{[t]}_{X,\F}
\]
which is a $\pi$-exceptional divisor. For each $\pi$-exceptional prime divisor $E_i$, the coefficient of $K^{[t]}_{Y,\F_Y}-\pi^*K^{[t]}_{X,\F}$ along $E_i$ is $(1-t)a(E_i,X)+t a(E_i,\F)$. Adding the boundary coefficient $(1-t)+t\varepsilon_{\F}(E_i)$ gives
\[
(1-t)\bigl(a(E_i,X)+1\bigr) + t\bigl(a(E_i,\F)+\varepsilon_{\F}(E_i)\bigr) = A^{[t]}_{X,\F}(E_i).
\]
Hence $D_{\pi}^{[t]}= \sum_i A^{[t]}_{X,\F}(E_i)E_i$.

Since
\[
K^{[t]}_{Y,\F_Y}+ \sum_i\bigl((1-t)+t\varepsilon_{\F}(E_i)\bigr)E_i
\]
is $\pi$-ample and $\pi^*K^{[t]}_{X,\F}$ is numerically trivial over $X$, the divisor $D_{\pi}^{[t]}$ is $\pi$-ample and $\pi$-exceptional. By the negativity lemma, its coefficients are non-positive.

We now show that they are in fact strictly negative. Suppose that some $\pi$-exceptional prime divisor $E_i$ had coefficient zero in $D_{\pi}^{[t]}$. Since $D_{\pi}^{[t]}$ is exceptional and has non-positive coefficients, we may choose a curve $C$ contracted by $\pi$, moving in $E_i$, and avoiding the other exceptional components. Then $D_{\pi}^{[t]}\cdot C=0$ contradicting the $\pi$-ampleness of $D_{\pi}^{[t]}$. Therefore every $\pi$-exceptional prime divisor appears in $D_{\pi}^{[t]}$ with strictly negative coefficient. Consequently $A^{[t]}_{X,\F}(E_i)<0$ for every $\pi$-exceptional prime divisor $E_i$, as required.
\end{proof}

\begin{theorem}
\label{thm:mixed-odaka-singularity-obstruction-normal}
Let $(X,\F,t)$ be a normal projective $\Q$-Gorenstein adjoint foliated structure with $\F$ algebraically integrable and $0<t<1$. Let $L$ be an ample $\Q$-Cartier divisor on $X$. If $(X,\F,L)$ is $t$-K-semistable, then $(X,\F,t)$ is log canonical.
\end{theorem}

\begin{proof}
We prove the contrapositive. Assume that $(X,\F,t)$ is not log canonical. Let $\pi:Y\to X$ be the relative mixed lc modification given by Lemma~\ref{lem:relative-mixed-lc-modification}, with $E_{\pi}=\sum_i E_i$ for the reduced exceptional divisor. By Lemma \ref{lem:relative-mixed-lc-modification}, every $\pi$-exceptional prime divisor satisfies $A^{[t]}_{X,\F}(E_i)<0$.

Since
\[
D_{\pi}^{[t]} := K^{[t]}_{Y,\F_Y} +\sum_i\bigl((1-t)+t\varepsilon_{\F}(E_i)\bigr)E_i-\pi^*K^{[t]}_{X,\F}= \sum_i A^{[t]}_{X,\F}(E_i)E_i
\]
is $\pi$-ample and exceptional with strictly negative coefficients, the divisor $A:=-D_{\pi}^{[t]}$ is effective and $-A$ is $\pi$-ample.

We now choose $m>0$ sufficiently divisible so that $mA$ is Cartier and the relative algebra
\[
\bigoplus_{k\ge0}\pi_*\cO_Y(-kmA)
\]
is generated in degree one. We now define the coherent ideal sheaf on $X$, $I:=\pi_*\cO_Y(-mA)\subset \cO_X$. Since $-mA$ is $\pi$-ample and the relative algebra $\bigoplus_{k\ge 0}\pi_*\cO_Y(-kmA)$ is generated in degree one, we have
\[
Y\simeq \operatorname{Proj}_X \bigoplus_{k\ge 0} I^k.
\]
Equivalently, $Y$ is the normalised blow-up of $I$. Hence the Rees valuations of $I$, at the generic points of the maximal-dimensional components of $\operatorname{Supp}(\cO_X/I)$, are precisely the valuations $\operatorname{ord}_{E_i}$ for the exceptional divisors $E_i$ lying over those components.

For every such relevant divisor $E_i$, we let $b_i:=\ord_{E_i}(I)>0$. After increasing $m$ if necessary, we may assume that all the $b_i$ are positive integers. We now choose $q\gg 0$ divisible by all the $b_i$ corresponding to the Rees divisors over the maximal-dimensional components of $\Supp(\cO_X/I)$. Let $\tau$ denote the coordinate on $\A^1$. We further choose $N\gg0$ and define the flag ideal
\[
\J:=\overline{(I+(\tau^q))^N} \subset \cO_{X\times\A^1},
\]
and let
\[
\Pi:\B:=\operatorname{Bl}_{\J}(X\times\A^1)^\nu \longrightarrow X\times\A^1
\]
be the normalised blow-up. Let also $\F_{\B}$ be the saturated birational transform of the product foliation. Let $G_\alpha$ be an exceptional divisor of $\Pi$ dominating a maximal-dimensional component of $\operatorname{Supp} \left(\cO_{X\times\A^1}/\J
\right)$.

We claim that for every such $G_\alpha$ of $\Pi$ there exists an index $i(\alpha)$, with $E_{i(\alpha)}$ lying over the corresponding maximal-dimensional component of $\Supp(\cO_X/I)$, such that $\ord_{G_\alpha}(\tau)=1$ and $\ord_{G_\alpha}|_{K(X)}
=
c_\alpha\ord_{E_{i(\alpha)}}$ for some $c_\alpha>0$.

This claim follows from the standard Rees-valuation calculation for the flag ideal $I+(\tau^q)$, which we recall for clarity. The assertion is local at the generic point of the relevant component. There, the top Rees valuations of $I$ are the valuations $\ord_{E_i}$. For a divisor over $X\times \A^1$ associated to $E_i$, we have as before $\ord_G|_{K(X)}=c\,\ord_{E_i}$ and $\ord_G(\tau)=r$. The two summands $I$ and $(\tau^q)$ have the same value along a Rees valuation of $I+(\tau^q)$, so $c\,\ord_{E_i}(I)=q\,r$. Equivalently, $c\,b_i=q\,r$. Because $q$ was chosen divisible by $b_i$, the primitive integral valuation corresponding to the associated exceptional divisor is obtained by taking $r=1$ and $c=\frac{q}{b_i}$. Thus, for every such $G_\alpha$, we have that $\ord_{G_\alpha}(\tau)=1$ and $ord_{G_\alpha}|_{K(X)} = \frac{q}{b_{i(\alpha)}}\ord_{E_{i(\alpha)}}$.

Raising $I+(\tau^q)$ to the $N$-th power and taking integral closure does not change the normalised blow-up or its Rees valuations, so the same conclusion holds for $\J=\overline{(I+(\tau^q))^N}$.

We now apply Lemma~\ref{lem:mixed-gauss-rees-coefficient}. Since $\ord_{G_\alpha}(\tau)=1$, the $(r-1)$-term in that lemma vanishes. Hence
\[
\operatorname{coeff}_{G_\alpha} K^{[t]}_{\B/(X\times\A^1)} = A^{[t]}_{X,\F} \left( \ord_{G_\alpha}|_{K(X)} \right).
\]
Using $\ord_{G_\alpha}|_{K(X)} = c_\alpha\ord_{E_{i(\alpha)}}$ and the homogeneity of mixed log discrepancy in the valuation, we get
\[
\operatorname{coeff}_{G_\alpha} K^{[t]}_{\B/(X\times\A^1)} =c_\alpha A^{[t]}_{X,\F}(E_{i(\alpha)}).
\]
Since every $A^{[t]}_{X,\F}(E_i)$ is strictly negative, every exceptional divisor $G_\alpha$ dominating a maximal-dimensional component of $\Supp(\cO_{X\times\A^1}/\J)$ has strictly negative coefficient in $K^{[t]}_{\B/(X\times\A^1)}$. The exceptional divisors $G_\alpha$ considered above are exactly the exceptional prime divisors $G_i$ whose image has dimension $s=\dim \operatorname{Supp}(\cO_{X\times \A^1}/\mathcal J)$. Furthermore, as in \cite[Lemma 7.12]{Pap26} $(\B,\F_\B,\L_r)$ is an $\F$-compatible (semi-) ample test configuration. Therefore, all the conditions of Proposition \ref{prop:mixed-odaka-negativity} are met, so we obtain
\[
S^{[t]}_{(X,\F),L}(\J)<0.
\]
Thus by Proposition \ref{prop:mixed-S-leading-term} $(X,\F,L)$ is not $t$-K-semistable, proving the contrapositive.
\end{proof}

\section{K-semistability characterisation for Calabi--Yau, general type and Fano adjoint foliated structures}\label{sec: odaka sing cond 2}

We now turn to the converse direction in the Calabi--Yau and general type cases.

\subsection{Reducing arbitrary test configurations to product birational models}
We begin by showing that starting with an arbitrary $\F$-compatible test configuration $\X$, we can find a common $\G_m$ equivariant model over $\overline{\X}$ and $X\times \PP^1$.

\begin{lemma}
\label{lem:common-model-f-compatible}
Let $(\X,\F_{\X},\L)\to \A^1$ be a normal $\F$-compatible test configuration for $(X,\F,L)$, and let $\overline{\X}\to\PP^1$ be its compactification. Then there exists a smooth projective $\G_m$-equivariant variety $\W$, together with birational morphisms
\[
\Theta:\W\to \overline{\X},
\qquad
\Pi:\W\to X\times\PP^1,
\]
such that both morphisms are isomorphisms over $\PP^1\setminus\{0\}$. Moreover, the saturated birational transform $\F_{\W}$ of
$\F_{\overline{\X}}$ coincides with the saturated birational transform of the product foliation on $X\times\PP^1$.
\end{lemma}

\begin{proof}
The test configuration gives a $\G_m$-equivariant birational map $X\times\PP^1\dashrightarrow \overline{\X}$ which is an isomorphism over $\PP^1\setminus\{0\}$. Let $\Gamma$ be the normalisation of the closure of its graph. Then $\Gamma$ comes with $\G_m$-equivariant birational morphisms
\[
\Gamma\to \overline{\X}, \qquad \Gamma\to X\times\PP^1.
\]
By the equivariant resolution of singularities, we can choose a smooth $\G_m$-equivariant resolution $\W\to\Gamma$. This gives the required morphisms $\Theta$ and $\Pi$.

Over $\PP^1\setminus\{0\}$, both $(\overline{\X},\F_{\overline{\X}})$ and $(X\times\PP^1,p_1^{-1}\F)$ are identified with the product family. Therefore the two induced foliations on the common open set agree. Their saturated extensions to $\W$ agree because saturation inside $T_{\W/\PP^1}$ is unique.
\end{proof}

We now show that the mixed Donaldson--Futaki invariant remains unchanged after taking the common model $\W$.

\begin{lemma}
\label{lem:birational-invariance-mixed-df}
Let $\cN:=\Theta^*\overline{\L}$. Then
\[
\DF^{[t]}(\W,\F_{\W},\cN) = \DF^{[t]}(\overline{\X}, \F_{\overline{\X}},\overline{\L}).
\]
\end{lemma}

\begin{proof}
Since $\cN=\Theta^*\overline{\L}$, we have $\cN^{n+1}=\overline{\L}^{n+1}$. Moreover,
\[
K_{\W/\PP^1} = \Theta^*K_{\overline{\X}/\PP^1}+R_X
\]
and
\[
K_{\F_{\W}} = \Theta^*K_{\F_{\overline{\X}}}+R_{\F},
\]
where $R_X$ and $R_{\F}$ are $\Theta$-exceptional. Hence
\[
K^{[t]}_{\W/\PP^1}=\Theta^*K^{[t]}_{\overline{\X}/\PP^1}+ R^{[t]},
\]
where $R^{[t]}:=(1-t)R_X+tR_{\F}$ is $\Theta$-exceptional. Therefore, by the projection formula,
\[
K^{[t]}_{\W/\PP^1}\cdot\cN^n=K^{[t]}_{\overline{\X}/\PP^1}\cdot\overline{\L}^{\,n},
\]
because $R^{[t]}\cdot\Theta^*\overline{\L}^{\,n}=0$. The slope and volume are unchanged since the general fibre is still $(X,\F,L)$. Substituting into the definition of the mixed Donaldson--Futaki invariant gives the desired equality.
\end{proof}

The following lemma rewrites the mixed Donaldson--Futaki invariant in terms of intersection numbers from specific divisors on the common model $\W$.

\begin{lemma}
\label{lem:common-model-odaka-decomposition}
Let $\M:=\Pi^*p_1^*L$ and $\cN:=\Theta^*\overline{\L}$. Then
\[
\DF^{[t]}(\W,\F_{\W},\cN)= \operatorname{Can}^{[t]}(\W,\cN) + \operatorname{Disc}^{[t]}(\W,\cN),
\]
where
\[
\operatorname{Can}^{[t]}(\W,\cN):=\frac{1}{V}
\left(\frac{n}{n+1}\mu(X,\F,L)\cN^{n+1}+ \Pi^*p_1^*K^{[t]}_{X,\F}\cdot\cN^n\right),
\]
and
\[
\operatorname{Disc}^{[t]}(\W,\cN):=\frac{1}{V}K^{[t]}_{\W/(X\times\PP^1)}\cdot \cN^n.
\]
\end{lemma}

\begin{proof}
Since $\F_{\W}$ is the saturated birational transform of the product foliation, we have
\[
K^{[t]}_{\W/\PP^1}=\Pi^*p_1^*K^{[t]}_{X,\F}+K^{[t]}_{\W/(X\times\PP^1)}.
\]
Substituting this into the definition of $\DF^{[t]}$ gives the claimed decomposition.
\end{proof}

We now show that $\cN = \Theta^*\overline{\L}$ is nef and can be expressed as a $\Q$-linear combination of $\M$ and $E$.

\begin{lemma}
\label{lem:line-bundle-comparison-flag}
After replacing $L$ and $\L$ by sufficiently divisible multiples, there exists an effective vertical $\Pi$-exceptional $\Q$-divisor $E$ on $\W$ such that $\cN\sim_{\Q}\M-E$. Moreover, $\cN$ is nef over $\PP^1$.
\end{lemma}

\begin{proof}
We apply \cite[Proposition~3.10]{Odaka12} to the underlying normal test configuration $(\X,\L)$. Since $\X$ is normal, it is in particular partially normal in the sense used there. Thus, after replacing $\L$ by a sufficiently divisible tensor power, there exists a $\G_m$-invariant flag ideal $\J\subset \cO_{X\times\A^1}$ and a blow-up
\[
\beta:\B:=\operatorname{Bl}_{\J}(X\times\A^1) \longrightarrow X\times\A^1
\]
which dominates $\X$. If
\[
\cO_{\B}(-E_{\J})
= \J\cdot\cO_{\B},
\]
then by \cite[Proposition 3.10]{Odaka12}, after replacing $L$ by the corresponding power, we obtain an equality of relatively semiample line bundles
\[
f^*\L\simeq \beta^*p_1^*L\otimes\cO_{\B}(-E_{\J}),
\]
where $f:\B\to\X$ is the induced birational morphism. Equivalently,
\[
f^*\L\sim_{\Q}\beta^*p_1^*L-E_{\J}.
\]

Now let $\Theta:\W\to\overline{\X}$, $\Pi:\W\to X\times\PP^1$ be the smooth equivariant common model chosen above. Replacing $\W$ by a higher equivariant common model if necessary, we may assume that $\W$ dominates the compactification of $\B$. Pulling the preceding equality back to $\W$, we obtain
\[
\cN:=\Theta^*\overline{\L} \sim_{\Q} \Pi^*p_1^*L-E,
\]
where $E$ is the pullback of $E_{\J}$ to $\W$. By definition, $\M:=\Pi^*p_1^*L$ and hence $\cN\sim_{\Q}\M-E$.

The divisor $E$ is effective because $E_{\J}$ is effective. It is vertical because the flag ideal $\J$ is supported on the central fibre $X\times\{0\}$. It is $\Pi$-exceptional, with the sign convention used above, because it is the exceptional divisor of the flag-ideal blow-up and its pullback to $\W$.

Finally, $\cN=\Theta^*\overline{\L}$ is nef over $\PP^1$, since $\overline{\L}$ is relatively nef over $\PP^1$ and nefness is preserved under pullback.
\end{proof}

We will now show how we can estimate the intersection numbers of $\cN$ and $\M$.

\begin{lemma}
\label{lem:odaka-canonical-part-inequality}
We have
\[
\cN^n\cdot(\M+nE)\ge0.
\]
Moreover, if the associated flag ideal is nontrivial, equivalently
$E\neq 0$ on the flag-ideal model, then
\[
N^n\cdot(M+nE)>0.
\]
\end{lemma}

\begin{proof}
This follows from \cite[Lemma 2.8]{Odaka11}. We briefly recall the argument for the reader's convenience. Since $\cN\sim_{\Q}\M-E$ we need to show
\[
(\M-E)^n\cdot(\M+nE)\ge0.
\]
Using the polynomial identity
\[
(x-y)^n(x+ny) =x^{n+1}- \sum_{i=1}^n (n+1-i)(x-y)^{n-i}x^{i-1}y^2,
\]
and substituting $x=\M$ and $y=E$, and using $\M^{n+1}=0$ because $\M$ is pulled back from the $n$-dimensional variety $X$, we obtain
\[
(\M-E)^n(\M+nE)=\sum_{i=1}^n(n+1-i)(-E^2)\cdot(\M-E)^{n-i}\M^{i-1}.
\]
Notice now that each term
\[
(-E^2)\cdot(\M-E)^{n-i}\M^{i-1}
\]
is nonnegative. To see this, after cutting by general members of $|mL|$ and by nef divisors representing $\M-E$, we reduce to a surface contained in a fibre of the exceptional locus, where the assertion follows from the Hodge index theorem and the relative ampleness of $-E$ over the flag-ideal centre. Hence every summand is nonnegative, and therefore
\[
(\M-E)^n\cdot(\M+nE)\ge0.
\]
The strict positivity statement follows from the same argument (cf. \cite[Lemma 2.8 and proof of Theorem 2.6]{Odaka11}).
\end{proof}

\subsection{Semistability in the adjoint Calabi--Yau and general type cases foliated cases}

We will compute the canonical part in the adjoint Calabi--Yau foliated case first, using the estimates we obtained in the previous subsection.

\begin{lemma}
\label{lem:canonical-part-mixed-CY}
Assume that $K^{[t]}_{X,\F}\equiv0$. Then, for every ample $\Q$-Cartier divisor $L$ and every common model as above,
\[
\operatorname{Can}^{[t]}(\W,\cN)=0.
\]
\end{lemma}

\begin{proof}
By numerical triviality,
\[
\mu(X,\F,L)=\frac{-K^{[t]}_{X,\F}\cdot L^{n-1}}{L^n}=0.
\]
Moreover, $\Pi^*p_1^*K^{[t]}_{X,\F}\equiv0$, therefore $\Pi^*p_1^*K^{[t]}_{X,\F}\cdot\cN^n=0$. Hence, both terms in the definition of $\operatorname{Can}^{[t]}$ presented in Lemma \ref{lem:common-model-odaka-decomposition} vanish, proving the claim.
\end{proof}

We will now do the same for the adjoint general type foliated case.

\begin{lemma}
\label{lem:canonical-part-mixed-general-type}
Assume that $K^{[t]}_{X,\F}$ is ample, and let $L=mK^{[t]}_{X,\F}$ for some sufficiently divisible positive integer $m$. Then, for every
common model as above,
\[
\operatorname{Can}^{[t]}(\W,\cN)\ge 0.
\]
Moreover, if the corresponding test configuration is non-trivial, then
\[
\operatorname{Can}^{[t]}(\W,\cN)>0.
\]
\end{lemma}

\begin{proof}
Since $K^{[t]}_{X,\F}=\frac{1}{m}L$, we have
\[
\mu(X,\F,L) = \frac{-K^{[t]}_{X,\F}\cdot L^{n-1}}{L^n}= -\frac{1}{m}.
\]
Moreover, $\Pi^*p_1^*K^{[t]}_{X,\F}=\frac{1}{m}\M$, and hence by Lemma \ref{lem:common-model-odaka-decomposition}
\[
\operatorname{Can}^{[t]}(\W,\cN)=\frac{1}{mV}\left(\M\cdot\cN^n-\frac{n}{n+1}\cN^{n+1}\right).
\]
By Lemma \ref{lem:line-bundle-comparison-flag} we have $\cN\sim_{\Q}\M-E$, so in particular $\cN^{n+1}=\cN^n\cdot(\M-E)$. Hence
\[
\M\cdot\cN^n-\frac{n}{n+1}\cN^{n+1}=\frac{1}{n+1}\cN^n\cdot(\M+nE).
\]
Thus
\[
\operatorname{Can}^{[t]}(\W,\cN)=\frac{1}{mV(n+1)} \cN^n\cdot(\M+nE)\geq 0.
\]
where the last inequality follows by Lemma \ref{lem:odaka-canonical-part-inequality}.

If the test configuration is non-trivial, then the associated flag ideal is nontrivial, equivalently $E\neq 0$ on the flag-ideal model. In this case, the strict form of the inequality in Lemma \ref{lem:odaka-canonical-part-inequality} gives $N^n\cdot(M+nE)>0$. Consequently $\operatorname{Can}^{[t]}(\W,\cN)>0$ completing the proof.
\end{proof}

The above computations allow us to deduce the K-semistability of Calabi--Yau and general type log canonical adjoint foliated structures, which will prove parts 1 and 2 of Theorem \ref{thm: main intro thm 2}. We start with the adjoint Calabi--Yau foliated case.

\begin{theorem}
\label{thm:mixed-odaka-CY}
Let $(X,\F,t)$ be a log canonical normal projective $\Q$-Gorenstein adjoint Calabi--Yau foliated structure. Then for every ample $\Q$-Cartier divisor $L$, the adjoint foliated structure $(X,\F,L)$ is $t$-K-semistable. If, furthermore $(X,\F,t)$ is klt, then for every ample $\Q$-Cartier divisor $L$, the adjoint foliated structure $(X,\F,L)$ is $t$-K-stable.
\end{theorem}

\begin{proof}
Let $(\X,\F_{\X},\L)\to\A^1$ be a normal $\F$-compatible test configuration. Using Lemma \ref{lem:common-model-f-compatible}, we pass to a smooth equivariant common model $\Theta:\W\to\overline{\X}$, $\Pi:\W\to X\times\PP^1$ and define, as before, $\cN:=\Theta^*\overline{\L}$. By Lemma \ref{lem:birational-invariance-mixed-df}, we have that the mixed Donaldson--Futaki invariants satisfy
\[
\DF^{[t]}(\X,\F_{\X},\L) = \DF^{[t]}(\W,\F_{\W},\cN).
\]
Furthermore, by Lemma \ref{lem:common-model-odaka-decomposition} these decompose as $\DF^{[t]} = \operatorname{Can}^{[t]}+\operatorname{Disc}^{[t]}$ where the canonical term vanishes $\operatorname{Can}^{[t]}=0$ by Lemma \ref{lem:canonical-part-mixed-CY}. Since $(X,\F,t)$ is log canonical, Corollary \ref{cor:mixed-relative-effective} gives that the divisor $K^{[t]}_{\W/(X\times\PP^1)}$ is effective. 

Furthermore, the divisor $K^{[t]}_{\W/(X\times\PP^1)}$ is vertical, and $\cN$ is relatively nef by Lemma \ref{lem:line-bundle-comparison-flag}. Therefore
\[
\operatorname{Disc}^{[t]}=\frac{1}{V}
K^{[t]}_{\W/(X\times\PP^1)}\cdot\cN^n \ge0.
\]
Thus, $\DF^{[t]}(\X,\F_{\X},\L)\ge0$. Since the test configuration was arbitrary, $(X,\F,L)$ is $t$-K-semistable.

We will now prove the case where $(X,\F,t)$ is klt. As before, the canonical part vanishes, so we need to show that the discriminant part is positive for a non-trivial $\F$-compatible test configuration. Taking the same common model $\W$ as before, recall that $K^{[t]}_{\mathcal W/(X\times\PP^1)} = \sum_j a_jG_j$ as a sum over $\Pi$-exceptional prime divisors. By Lemma \ref{lem:mixed-gauss-rees-coefficient} we have $a_j= A^{[t]}_{X,\F}(v_j)+(1-t)(r_j-1)$ for $v_j:=\ord_{G_j}|_{K(X)}$ and $r_j:=\ord_{G_j}(s)$.

Since $(X,\F,t)$ is klt, we have $A^{[t]}_{X,\F}(v_j)>0$. Moreover $r_j\ge1$, because $G_j$ lies over the central fibre. Hence $a_j>0$ for every $\Pi$-exceptional prime divisor $G_j$.

Since $E$ is an effective vertical exceptional divisor supported on the exceptional locus of $\Pi$, there exists a rational number $c>0$ such that $K^{[t]}_{\mathcal W/(X\times\PP^1)}-cE$ is effective. Because $\cN$ is relatively nef by Lemma \ref{lem:line-bundle-comparison-flag} and the divisor above is vertical and effective, we obtain
\[
K^{[t]}_{\mathcal W/(X\times\PP^1)}\cdot\cN^n \ge c\, E\cdot\cN^n.
\]
Thus it is enough to show that $E\cdot\cN^n>0$. First, we have
\[
\cN^{n+1} = (\M-E)^{n+1} = (\M-E)^{n+1}-\M^{n+1} \le0,
\]
by \cite[Lemma 2.8]{Odaka11}. On the other hand, as in \cite[Proof of Theorem 2.6]{Odaka11}, Lemma \ref{lem:odaka-canonical-part-inequality} gives $\cN^n\cdot(\M+nE)>0$ for a nontrivial flag ideal. In particular, since $\M+nE =\cN+(n+1)E$, we have $\cN^n\cdot(\M+nE) = \cN^{n+1} + (n+1)\cN^n\cdot E$, which is strictly positive and $\cN^{n+1}\le0$. This implies that $\cN^n\cdot E>0$.

Therefore $K^{[t]}_{\mathcal W/(X\times\PP^1)} \cdot\cN^n>0$, and hence $\DF^{[t]}(\X,\F_{\X},\L) >0$ for every non-trivial $\F$-compatible test configuration. Thus $(X,\F,L)$ is $t$-K-stable.
\end{proof}

We will now prove the adjoint general type foliated case.

\begin{theorem}
\label{thm:mixed-odaka-general-type}
Let $(X,\F,t)$ be a normal log canonical projective $\Q$-Gorenstein adjoint general type foliated structure. Then, for every sufficiently divisible positive integer $m$, the adjoint general type foliated structure $(X,\F,mK^{[t]}_{X,\F})$ is $t$-K-stable.
\end{theorem}

\begin{proof}
We let $L:=mK^{[t]}_{X,\F}$ and $(\X,\F_{\X},\L)\to\A^1$ be a normal $\F$-compatible test configuration for $(X,\F,L)$. As in the proof of Theorem \ref{thm:mixed-odaka-CY} we pass to a smooth equivariant common model $\Theta:\W\to\overline{\X}$, $\Pi:\W\to X\times\PP^1$ and define $\cN:=\Theta^*\overline{\L}$. By Lemma \ref{lem:birational-invariance-mixed-df}, the mixed Donaldson--Futaki invariants satisfy
\[
\DF^{[t]}(\X,\F_{\X},\L) = \DF^{[t]}(\W,\F_{\W},\cN).
\]
and decompose as $\DF^{[t]} = \operatorname{Can}^{[t]}+\operatorname{Disc}^{[t]}$ by Lemma \ref{lem:common-model-odaka-decomposition}.

Furthermore, Lemma \ref{lem:canonical-part-mixed-general-type} shows that $\operatorname{Can}^{[t]}\geq 0$. Since $(X,\F,t)$ is log canonical, Corollary \ref{cor:mixed-relative-effective} shows that $K^{[t]}_{\W/(X\times\PP^1)}$ is effective. Thus, as in the proof of Theorem \ref{thm:mixed-odaka-CY}, we have
\[
\operatorname{Disc}^{[t]} = \frac{1}{V} K^{[t]}_{\W/(X\times\PP^1)} \cdot\cN^n \ge0.
\]
Therefore $\DF^{[t]}(\X,\F_{\X},\L)\ge 0$ for every normal $\F$-compatible test configuration.

If the test configuration is non-trivial, then the associated flag ideal is nontrivial, and Lemma \ref{lem:canonical-part-mixed-general-type} gives $\operatorname{Can}^{[t]}(W,N)>0$. Since the discrepancy part is nonnegative, we obtain $\DF^{[t]}(\X,\F_{\X},\L)>0$. Hence $(X,\F,mK^{[t]}_{X,\F})$ is $t$-K-stable.
\end{proof}

\begin{corollary}\label{cor: log boundedness gt}
    Let $d$ be a positive integer, $\Gamma\subset (0,1)$ a set satisfying the descending chain condition (DCC), and $C$ a positive real number. Then the following set of projective algebraically integrable foliated pairs
    $$\bigg\{(X,\F)\bigg|\dim(X) =d,\text{ } \exists t\in \Gamma \text{ s.t. } (X,\F,t) \text{ is K-stable general type, and } 0<\vol(K_{X,\F}^{[t]})\leq C \bigg\}$$
    is log birationally bounded.
\end{corollary}
\begin{proof}
    By Theorems \ref{thm:mixed-odaka-singularity-obstruction-normal} and \ref{thm:mixed-odaka-general-type} an adjoint general type foliated structure $(X,\F,t)$ is $t$-K-stable (for $L = K_{X,\F}^{[t]}$) if and only if $(X,\F,t)$ is log canonical. The result then follows from \cite[Theorem B]{CLSV26}.
\end{proof}

\subsection{Semistability in the adjoint Fano foliated case}

We now focus on adjoint Fano foliated structures, which will prove part 3 of Theorem \ref{thm: main intro thm 2}. 

\begin{theorem}\label{thm:mixed-odaka-Fano}
Let $(X,\F,t)$ be a normal projective $\Q$-Gorenstein adjoint Fano foliated structure. Let $L:=-mK^{[t]}_{X,\F}$ for some sufficiently divisible positive integer $m$. If $(X,\F,L)$ is $t$-K-semistable, then $(X,\F,t)$ is klt. In particular, $X$ is potentially klt, and if it is $\Q$-factorial it is klt. Furthermore, $X$ is of Fano type.
\end{theorem}

\begin{proof}
By Theorem \ref{thm:mixed-odaka-singularity-obstruction-normal}, $t$-K-semistability implies that $(X,\F,t)$ is log canonical. Suppose, for contradiction, that $(X,\F,t)$ is not klt.

Since $0<t<1$, we may apply the qdlt modification theorem for algebraically integrable adjoint foliated structures. More precisely, \cite[Theorem~1.9]{CHLMSSX24} gives a projective birational morphism $h:X'\to X$ such that, if $\F':=h^{-1}\F$ and
\[
B'=\operatorname{Exc}(h)^{\rm ninv}+(1-t)\operatorname{Exc}(h)^{\rm inv},
\]
then $(X',\F',B',0,t)$ is $\Q$-factorial lc. Moreover, every $h$-exceptional prime divisor is an nklt place of $(X,\F,t)$, and since $(X,\F,t)$ is lc, every such divisor has mixed log discrepancy equal to zero. In addition, because the original structure is lc, the qdlt modification is crepant:
\[
K(X',\F',B',0,t)=h^*K(X,\F,0,0,t).
\]
Here we use the notation of \cite[Theorem~1.9]{CHLMSSX24}.

Because $(X,\F,t)$ is not klt, the non-klt locus is nonempty. Since $0<t<1$ and there is no boundary, a prime divisor lying on $X$ has positive mixed log discrepancy; hence the non-klt locus is detected by exceptional divisors over $X$. Thus, after possibly replacing $h$ by the qdlt modification over the relevant component, we may assume that $h$ is not an isomorphism. In particular, we can take $h$ to be the qdlt modification extracting the relevant non-klt places. Let $E_1,\ldots,E_\ell$ be the $h$-exceptional prime divisors. In particular, $A^{[t]}_{X,\F}(E_i)=0$.

We now choose an effective $h$-exceptional $\Q$-divisor $A'$ such that $-A'$ is $h$-ample. After replacing $A'$ by a sufficiently divisible multiple, we take $m>0$ so that $mA'$ is Cartier and the relative algebra
\[
\bigoplus_{k\ge 0}h_*\cO_{X'}(-kmA')
\]
is generated in degree one. We define the ideal $I:=h_*\cO_{X'}(-mA')\subset \cO_X$ where $X'\simeq \operatorname{Bl}_I(X)^\nu$. In particular, the Rees valuations of $I$ are precisely the divisorial valuations $\ord_{E_i}$ associated to the $h$-exceptional prime divisors. For each $i$, we let $b_i:=\ord_{E_i}(I)>0$, and we choose $q>0$ divisible by all the $b_i$. We also choose $N\gg 0$, and consider the ideal $\J:=\overline{(I+(s^q))^N} \subset \cO_{X\times\A^1}$ and the normalised blow-up, 
\[
\Pi:\B:=\operatorname{Bl}_{\J}(X\times\A^1)^\nu
\to X\times\A^1.
\]
with saturated birational transform of the product foliation $\F_\B$. Recall that as in \cite[Lemma 7.12]{Pap26}, $(\B,\F_\B,\L_r)$, with $\L_r:=r\M-E$, and $\M:=\Pi^*p_1^*L$, is a (semi-)ample test configuration.

Raising $I+(\tau^q)$ to the $N$-th power does not change the normalised blow-up or its Rees valuations. Let $G_\alpha$ be a $\Pi$-exceptional prime divisor, and let $w_\alpha:=\ord_{G_\alpha}$. Since $w_\alpha$ is a Rees valuation of $I+(\tau^q)$, its restriction to $K(X)$ is a positive multiple of one of the Rees valuations of $I$, i.e. $w_\alpha|_{K(X)}=c_\alpha\ord_{E_{i(\alpha)}}$. Moreover, along a Rees valuation of $I+(\tau^q)$, the two summands $I$ and $(\tau^q)$ have the same value. Therefore $c_\alpha b_{i(\alpha)}=q\,w_\alpha(\tau)$. By the choice of $q$, and after taking the primitive integral normalisation of the divisorial valuation $w_\alpha=\ord_{G_\alpha}$, we have $w_\alpha(\tau)=1$ and $c_\alpha=\frac{q}{b_{i(\alpha)}}$. Thus every $\Pi$-exceptional prime divisor $G_\alpha$ satisfies $\ord_{G_\alpha}(\tau)=1$ and $\ord_{G_\alpha}|_{K(X)}=c_\alpha\ord_{E_{i(\alpha)}}$ for some $c_\alpha>0$.

By applying Lemma \ref{lem:mixed-gauss-rees-coefficient}, we obtain
\[
\operatorname{coeff}_{G_\alpha} K^{[t]}_{\B/(X\times\A^1)} = A^{[t]}_{X,\F} \left(\ord_{G_\alpha}|_{K(X)}\right) = c_\alpha A^{[t]}_{X,\F}(E_{i(\alpha)}) = 0.
\]
Since $A^{[t]}_{X,\F}(E_i)=0$ for every $i$ every coefficient of $K^{[t]}_{\mathcal B/(X\times\A^1)}$ along a $\Pi$-exceptional prime divisor is zero. Since this relative mixed canonical divisor is $\Pi$-exceptional, it follows that $K^{[t]}_{\mathcal B/(X\times\A^1)}=0$. After compactifying over $\PP^1$, the same equality gives $K^{[t]}_{\mathcal B/(X\times\PP^1)}=0$ since the flag-ideal modification is trivial over $\PP^1\setminus\{0\}$. Therefore the discrepancy part of the mixed Donaldson--Futaki invariant of the flag-ideal test configuration $(\B,\F_{\B},\L_r)$ vanishes, i.e. $\operatorname{Disc}^{[t]}(\B,\L_r)=0$.

We now use the anti-adjoint polarisation. Since $L=-mK^{[t]}_{X,\F}$, the canonical part of the mixed Donaldson--Futaki invariant has the
opposite sign from the canonically polarised case which was analysed in Lemma \ref{lem:canonical-part-mixed-general-type}. More concretely, recall that $\cO_{\B}(-E)=\J\cdot\cO_{\B}$. Since $L=-mK^{[t]}_{X,\F}$, we have $K^{[t]}_{X,\F}=-\frac{1}{m}L$ and $\mu(X,\F,rL)=\frac{1}{mr}$. Thus the canonical part is
\[
\operatorname{Can}^{[t]}_r = -\frac{1}{m r^{n+1}V(n+1)} \L_r^n\cdot(r\M+nE).
\]
The flag ideal is nontrivial, so Lemma \ref{lem:odaka-canonical-part-inequality} gives $\L_r^n\cdot(\M+nE)>0$. Therefore the canonical part is strictly negative.

Since the discrepancy part vanishes, for sufficiently large exponent in the flag-ideal construction, we obtain $\operatorname{DF}^{[t]}(\B,\F_{\B},\L_r)<0$ for the $\F$-compatible test configuration $(\B,\F_{\B},\L_r)$. Hence $(X,\F,t)$ must be $t$-K-unstable, contradicting the assumed $t$-K-semistability of $(X,\F,L)$. Hence no divisor with zero mixed log discrepancy exists. Since $(X,\F,t)$ is already log canonical by Theorem \ref{thm:mixed-odaka-singularity-obstruction-normal}, it follows that $(X,\F,t)$ is klt. 

To conclude, by \cite[Theorem 1.10(1)]{CHLMSSX24}, since $(X,\F,t)$ is klt, the variety $X$ is potentially klt. The statement on $\Q$-factoriality follows from the definition of potentially klt (see also \cite[Proposition 9.1]{CHLMSSX25}). The fact that $X$ is of Fano type follows from \cite[Theorem C]{CHLMSSX25}.
\end{proof}

\printbibliography

@article{Odaka12,
  author    = {Yuji Odaka},
  title     = {A generalization of the Ross--Thomas slope theory},
  journal   = {Osaka Journal of Mathematics},
  volume    = {50},
  number    = {1},
  year      = {2013},
  pages     = {171--185}
}

@article{Odaka13,
  author  = {Odaka, Yuji},
  title   = {The {GIT} stability of polarized varieties via discrepancy},
  journal = {Annals of Mathematics},
  volume  = {177},
  number  = {2},
  year    = {2013},
  pages   = {645--661},
  doi     = {10.4007/annals.2013.177.2.6}
}

@article{CHLMSSX25,
  author = {Cascini, Paolo and Han, Jingjun and Liu, Jihao and Meng, Fanjun and
            Spicer, Calum and Svaldi, Roberto and Xie, Lingyao},
  title = {On finite generation and boundedness of adjoint foliated structures},
  eprint = {2504.10737},
  archivePrefix = {arXiv},
  primaryClass = {math.AG},
  year = {2025}
}

@article{CHLMSSX24,
  author       = {Paolo Cascini and Jingjun Han and Jihao Liu and Fanjun Meng and Calum Spicer and Roberto Svaldi and Lingyao Xie},
  title        = {Minimal model program for algebraically integrable adjoint foliated structures},
  year         = {2024},
  eprint       = {2408.14258},
  archivePrefix= {arXiv},
  primaryClass = {math.AG},
  url          = {https://arxiv.org/abs/2408.14258}
}

@book{Xu2025,
  author    = {Xu, C.},
  title     = {K-stability of Fano Varieties},
  series    = {New Mathematical Monographs},
  volume    = {50},
  publisher = {Cambridge University Press},
  year      = {2025},
  isbn      = {9781009538770},
  doi       = {10.1017/9781009538763}
}

@book{Brunella15,
  author    = {Marco Brunella},
  title     = {Birational Geometry of Foliations},
  series    = {IMPA Monographs},
  volume    = {1},
  publisher = {Springer},
  year      = {2015},
  doi       = {10.1007/978-3-319-14310-1}
}

@article{AD13,
  author    = {Carolina Araujo and St{\'e}phane Druel},
  title     = {On Fano foliations},
  journal   = {Advances in Mathematics},
  volume    = {238},
  year      = {2013},
  pages     = {70--118},
  doi       = {10.1016/j.aim.2013.01.003}
}

@article{Odaka11,
  author  = {Odaka, Yuji},
  title   = {The Calabi conjecture and {$K$}-stability},
  journal = {Int. Math. Res. Not. IMRN},
  year    = {2012},
  number  = {10},
  pages   = {2272--2288},
  doi     = {10.1093/imrn/rnr159}
}

@article{CDS2013, 
title={K\"{a}hler–{E}instein Metrics and Stability}, 
volume={2014},
DOI={10.1093/imrn/rns279},
number={8},
journal={International Mathematics Research Notices}, 
author={Chen, Xiuxiong and Donaldson, Simon and Sun, Song},
year={2013},
pages={2119–2125}
}

@article{ACSS21,
  author       = {Florin Ambro and Paolo Cascini and Vyacheslav V. Shokurov and Calum Spicer},
  title        = {Positivity of the moduli part},
  year         = {2021},
  eprint       = {2111.00423},
  archivePrefix= {arXiv},
  primaryClass = {math.AG},
  doi          = {10.48550/arXiv.2111.00423},
  url          = {https://arxiv.org/abs/2111.00423}
}

@article{CS25,
  author    = {Paolo Cascini and Calum Spicer},
  title     = {On the {MMP} for rank one foliations on threefolds},
  journal   = {Forum of Mathematics, Pi},
  volume    = {13},
  year      = {2025},
  pages     = {e20},
  doi       = {10.1017/fmp.2025.10013}
}

@misc{Pap26,
      title={K-stability of adjoint foliated structures}, 
      author={Theodoros Stylianos Papazachariou},
      year={2026},
      eprint={2605.21995},
      archivePrefix={arXiv},
      primaryClass={math.AG},
      url={https://arxiv.org/abs/2605.21995}, 
}

@misc{CLSV26,
      title={Birational boundedness of stable families}, 
      author={Paolo Cascini and Jihao Liu and Calum Spicer and Roberto Svaldi},
      year={2026},
      eprint={2604.24106},
      archivePrefix={arXiv},
      primaryClass={math.AG},
      url={https://arxiv.org/abs/2604.24106}, 
}

\end{document}